\documentclass[aip,amsmath,amssymb,reprint,floatfix]{revtex4-1}

\usepackage{graphicx}
\usepackage{dcolumn}
\usepackage{bm}
\usepackage{amsthm}

\usepackage[utf8]{inputenc}
\usepackage[T1]{fontenc}
\usepackage{mathptmx}
\usepackage{etoolbox}
\newtheorem{theorem}{Theorem}
\newtheorem{definition}{Definition}
\newtheorem{lemma}{Lemma}
\newtheorem{remark}{Remark}

\newtheorem*{corollary}{Corollary}
\makeatother
\begin{document}

\preprint{AIP/123-QED}

\title[Analyzing Topological Mixing and Chaos on Continua with Symbolic Dynamics]{Analyzing Topological Mixing and Chaos on Continua with Symbolic Dynamics}
\author{Arnaldo Rodriguez-Gonzalez}
\email{ajr295@cornell.edu}
\noaffiliation
\affiliation{Sibley School of Mechanical \& Aerospace Engineering, Cornell University, Ithaca NY}

\date{\today}

\begin{abstract}
This work describes the way that topological mixing and chaos in continua, as induced by discrete dynamical systems, can or can't be understood through topological conjugacy with symbolic dynamical systems. For example, there is no symbolic dynamical system that is topologically conjugate to any discrete dynamical system on an entire continuum, and there is no finer topology that can be given to such a continuum system which can make it topologically conjugate to a symbolic dynamical system. However, this paper demonstrates an analytical mechanism by which the existence of topological mixing and/or chaos can be shown through conjugacy with qualitative dynamical systems outside the usual purview of symbolic dynamics. Two examples of this mechanism are demonstrated on classic textbook models of chaotic dynamics; the first proving the existence of topological mixing everywhere in the dyadic map on the interval by showing that there exists a qualitative system that is topologically conjugate to the dyadic map on the interval with a finer topology than the usual Euclidean topology, and the other following a similar approach to demonstrate the existence of Devaney chaos everywhere in the $2$-tent map on the interval. The content is presented in a somewhat self-contained fashion, reiterating some standard results in the field, to aid new learners of topological mixing/chaos.
\end{abstract}
\maketitle

\section{\label{sec:intro}Introduction}

The study of mixing on continua is an important application of dynamical systems theory; proving the existence of topological transitivity/mixing is a key part of the analysis of mixing in fluids\cite{ottino1989kinematics}, often by invoking dynamical equivalence through topological conjugacy to celebrated examples of dynamical systems that are chaotic within a continuum such as the dyadic map, tent map, or logistic map for certain parameter values\cite{Bringer2004}. Similarly, establishing a topological conjugacy between a modeled system and a symbolic dynamical system has proven to be an extraordinarily fruitful mechanism for diagnosing the existence of topological mixing and other chaotic behaviors inside of dynamical systems, with examples including the confirmation of chaotic behavior in the Newtonian three-body problem\cite{koon2006dynamical}, in pendulum-like systems\cite{zehnder2010lectures,pendulumthing}, and in general systems exhibiting hyperbolic dynamics\cite{guckenheimer2014nonlinear}.

That said, most examples of applying this conjugacy between a system exhibiting chaos/mixing somewhere within a continuum and a symbolic dynamical system rely on merely establishing semi-conjugacies, which fails to retain a one-to-one correspondence between the orbits in the continuum system and the qualitative description of those orbits in the symbolic system. Another challenge to the traditional approach is that the regions in which this chaos/mixing can be proved to occur, as prescribed by conjugacy to the symbolic system, are always infinitesimally "small" due to topological constraints induced by the conjugacy. An example of this is shown in a canonical chaotic dynamical system—--the $3$-tent map---in Figure \ref{fig:tentmap3}. This notion of smallness is formalized by the notion of punctiformity\cite{steen2013counterexamples}, which is discussed in Section \ref{sec:2} of this text.

\begin{figure}
\includegraphics[width=0.47\textwidth,keepaspectratio]{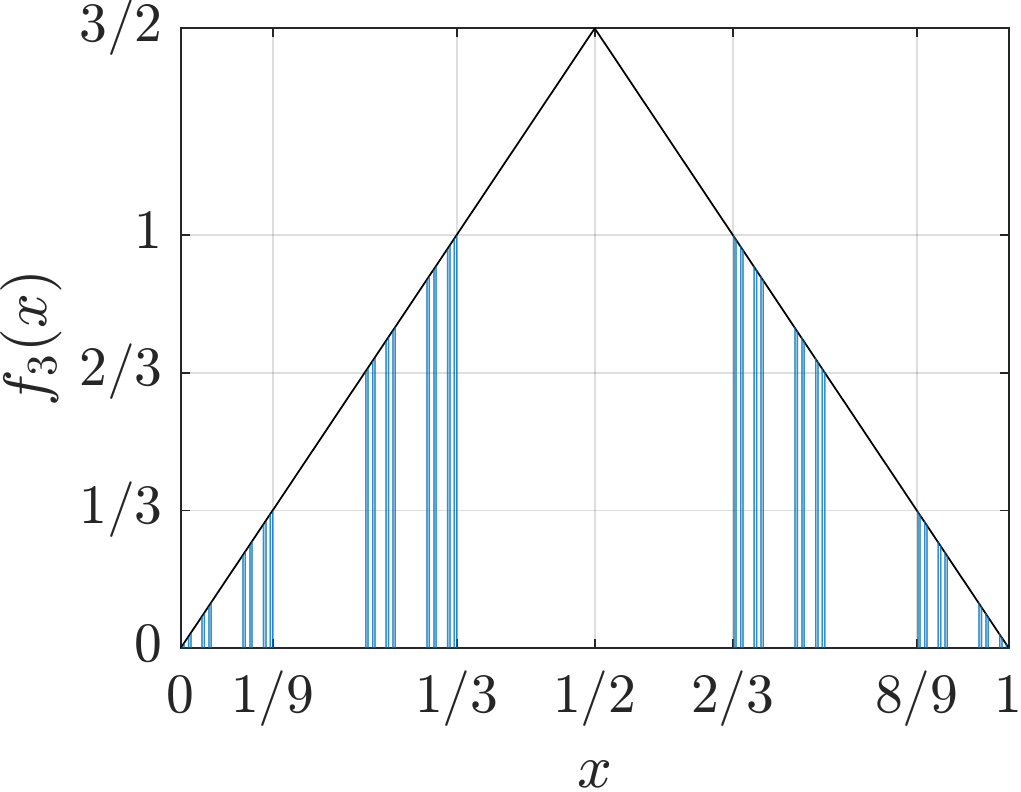}
\caption{\label{fig:tentmap3} A graph of the $3$-tent map $f_3$ on the unit interval $[0,1]$. The vertical lines indicates the values of $x\in [0,1]$ for which the mapping is mixing \& chaotic; all other elements in $x$ are eventually mapped outside the unit interval. Therefore, for almost all $x\in [0,1]$, the map does not exhibit mixing or chaos (as defined in Section \ref{topconj} of this paper); the subspace of the interval on which these behaviors do occur is the standard Cantor set, and is therefore punctiform. For proofs of this and further discussion on the $3$-tent map, see Banks \& Dragan\cite{BanksDragan}.}
\end{figure}

This paper describes a method that allows a dynamicist to prove the existence of chaos/mixing everywhere in a continuum by altering the continuum's topology to be finer, and establishing a topological conjugacy between the altered system and a qualitative dynamical system, avoiding both of the analytical pitfalls mentioned previously. The method is then demonstrated on two examples on classic textbook systems in Sections \ref{sec:mixing} and \ref{chaos}. The following discussions assume the reader is familiar with topology and topological properties; see Appendix A for a primer on these topics. Also, this text focuses exclusively on topological/metric notions of chaos; their measure-theoretic analogues\cite{RevModPhys.57.617} are not discussed herein.
\subsection{\label{sec:level2}Qualitative \& symbolic dynamical systems}
A broad definition of a qualitative dynamical system will be mostly useful in the latter part of this text, and is introduced to contrast with the formal definition of a symbolic dynamical system below, which is identical to that found in most specialized literature on the subject\cite{lind_marcus_1995,Kitchens1998,wiggins_2010}:
\begin{definition}
\label{def:1}
A qualitative dynamical system $(Q,\ell,\sigma)$ consists of a 3-tuple:
\begin{enumerate}
    \item A set $Q$ of infinite symbol sequences, composed from a finite set of symbols $\Sigma$.
    \item A metric $\ell$ defining the distance between elements of $Q$.
    \item The shift operator $\sigma$ which, upon acting on an element of $Q$, "deletes" the first symbol of the infinite symbol sequence and shifts the position of the remaining symbols "forward".
\end{enumerate}
In addition, the metric space $(Q,\ell)$ must always have the metric topology that is equal to the product topology of $\Sigma^\mathbb{N}$.
\newline
\newline
This qualitative dynamical system is also considered a symbolic dynamical system, or subshift\cite{berthe_rigo_2016}, if and only if $(Q,\ell)$ under this topology is compact and $Q$ is shift-invariant ($\sigma(Q) = Q$).
\end{definition}
A more straightforward way of visualizing the topology of the underlying metric spaces described above is by considering that every subset of elements of $Q$ that share the same "prefix" sequence is an open set, and that this type of open set forms a basis for the metric topology of $(Q,\ell)$; they are commonly referred to as the cylinder sets\cite{lind_marcus_1995,Kitchens1998,robinson1998dynamical} of this topology.

The following topological property of qualitative dynamical systems is also important:
\begin{remark}
If $(Q,\ell,\sigma)$ is a qualitative dynamical system, $(Q,\ell)$ is a totally disconnected metric space.
\label{rem:1}
\end{remark}
\begin{proof}Given a subset $Y \subseteq Q$, consider any element $y \in Y$. There must exist some integer $n$ such that the prefix of length $n$ of $y$ is not shared by any other element of $Y$. The cylinder set $U$ of all elements in $Q$ that share that prefix, and the union $V$ of the cylinder sets representing every other prefix of length $n$ present in $X$, are disjoint and contain $Y$ in their union. Since this is true for all subsets of $Q$, $(Q,\ell)$ is a totally disconnected metric space.
\end{proof}
\subsection{\label{Continua}Dynamical systems on continua}
The discussion in this paper is focused on discrete dynamical systems defined over continua; in a way, they are analogues for models of dynamical systems that are "continuous" in space but discrete in time. This can be made rigorous by defining the notion of a continuum, and a discrete dynamical system, precisely below:
\begin{definition}
\label{def:2}
A continuum is a compact, connected metric space with uncountably many elements.
\end{definition}
\begin{definition}
\label{def:3}
A discrete dynamical system $(X,d,f)$ consists of a 3-tuple:
\begin{enumerate}
    \item An underlying set $X$.
    \item A metric $d$ defining the distance between elements of $X$.
    \item A function $f: X\to X$ that "evolves" the elements of $X$.
\end{enumerate}
Notice that qualitative dynamical systems, as defined in Definition \ref{def:1}, are automatically discrete dynamical systems.
\end{definition}
By natural extension, a discrete dynamical system on a continuum is simply a discrete dynamical system $(X,d,f)$ where the underlying metric space $(X,d)$ is a continuum as described in Definition \ref{def:2}. 
\subsection{\label{topconj}Mixing, chaos, and topological conjugacy}
Having discussed what the dynamical systems of interest for this text are, it follows naturally to describe the notion of mixing and chaos rigorously within them. Any construction of mixing should follow the intuitive notion that the "contents" of any two "regions" in the system overlap after some finite amount of time; by describing "regions" in the system as open sets in a metric space, this idea is precisely what the notions of topological transitivity and topological mixing aim to convey.
\begin{definition}
\label{def:4}
A discrete dynamical system $(X,d,f)$ is topologically transitive if and only if, for every pair of open sets $U$ and $V$ in the metric topology of $(X,d)$, there exists an integer $n\geq 0$ such that $f^n(U) \cap V \neq \emptyset$.
\end{definition}
\begin{definition}
\label{def:5}
A discrete dynamical system $(X,d,f)$ is topologically mixing if and only if, for every pair of open sets $U$ and $V$ in the metric topology of $(X,d)$, there exists an integer $n\geq 0$ such that $f^N(U) \cap V \neq \emptyset$ for all integer $N \geq n$.
\end{definition}
Transitivity is widely considered one of the hallmarks of chaotic behavior, perhaps second only to the notion of sensitivity---the idea that arbitrarily close elements in a metric space eventually separate past a specific distance.
\begin{definition} 
\label{def:6}
Consider a discrete dynamical system $(X,d,f)$. This system is sensitive if and only if, for every element $x \in X$ and any non-negative number $\delta$, there exists another element $y \in X$ and a positive integer $n$ such that $d(x,y) < \delta$ but $d(f^n(x),f^n(y)) > \epsilon$ for some non-negative real number $\epsilon$.
\end{definition}
There is one other property that is sometimes considered a requirement for chaos, although it is less "intuitively chaotic" than the others; the idea that there are periodic orbits "everywhere" in the system.
\begin{definition}
\label{def:7} Consider a discrete dynamical system $(X,d,f)$. This system has dense periodic orbits if and only if, in every open set $U$ in the metric topology of $(X,d)$, there exists an element $x \in U$ such that $f^n(x) = x$ for some positive integer $n$.
\end{definition}
These three properties in conjunction are what we will use to define chaos:
\begin{definition}
\label{def:8}
    Consider a discrete dynamical system $(X,d,f)$. $(X,d,f)$ is chaotic, or Devaney chaotic\cite{Devaney2021}, if and only if it is topologically transitive, sensitive, and has dense periodic orbits.
\end{definition}
Although some authors prefer to impose further or alternative restrictions for chaoticity (for example, requiring $(X,d)$ to be compact\cite{wiggins_2010}), the definition here is chosen to be as general as reasonably possible with the understanding that some systems which are chaotic in this sense may not exhibit intuitively chaotic behavior. 

One of the reasons that the idea of dense periodic orbits is considered in the study of chaos is because of the following convenient result:
\begin{theorem}[Banks-Brooks-Cairns-Davis-Stacey Theorem\cite{Banks1992}]
\label{thm:1}Suppose $(X,d,f)$ was a discrete dynamical system and $f$ is continuous on $(X,d)$. If $(X,d,f)$ is topologically transitive and has dense periodic orbits, then $(X,d,f)$ is sensitive.
\end{theorem}
The Banks-Brooks-Cairns-Davis-Stacey theorem, under the appropriate continuity restriction on $f$, allows us to replace the requirement of sensitivity on our systems with that of dense periodic orbits, as the latter (in conjunction with transitivity) implies the former in such systems.

Finally, we can formalize the idea that two discrete dynamical systems are "dynamically equivalent" through the standard notion of topological conjugacy:
\begin{definition}
\label{def:9}Two discrete dynamical systems $(X_1,d_1,f_1)$ and $(X_2,d_2,f_2)$ are topologically conjugate if and only if a homeomorphism $h: X_1\to X_2$ exists such that $h\circ f_1 = f_2 \circ h$.
\newline

Similarly, $(X_1,d_1,f_1)$ is topologically semi-conjugate to $(X_2,d_2,f_2)$ if and only if a continuous surjection $s: X_1\to X_2$ exists such that $s \circ f_1 = f_2\circ s$.
\end{definition}
The conjugacy described by $h$ is very strong; orbits in $(X_1,d_1,f_1)$ would have the same periodicity as their counterparts in $(X_2,d_2,f_2)$, and the underlying metric spaces $(X_1,d_1)$ and $(X_2,d_2)$ in these systems have equivalent topologies.

The following pragmatically important remark concludes this discussion:
\begin{remark}
\label{rem:2}
If a discrete dynamical system is topologically transitive, topologically mixing, or has dense periodic orbits, any discrete dynamical system which is topologically conjugate to it is also transitive/mixing/has dense periodic orbits.
\end{remark}
The transitivity/mixing portion of this remark follows from the fact that the homeomorphism implied between spaces in the notion of topological conjugacy preserves set operations; if the intersection of two open sets in one of the metric spaces is nonempty, the intersection of their images under the conjugacy-induced homeomorphism must also be nonempty. And since conjugacy preserves periodicity, the one-to-one correspondence between the topologies implied by the homeomorphism indicates that, if every open set in one metric space contains a periodic orbit, so must every open set in the conjugated topology.

Practically, this remark reveals a valuable tool for proving the existence of mixing or chaos in a discrete dynamical system; showing its existence in a simpler-to-understand qualitative dynamical system, and then proving topological conjugacy between the qualitative system and the original system. It also highlights the utility of the Banks-Brooks-Cairns-Davis-Stacey theorem; although sensitivity is not always preserved by topological conjugacy, the existence of dense periodic orbits is, providing a general mechanism to diagnose chaos via conjugacy when the restriction of continuity of $f$ is present in such systems.
\section{\label{sec:2}Describing Continuum Systems with Symbolic Dynamics}
Given the utility of symbolic dynamical models to describe discrete dynamical systems as described \& cited in the introduction of this text, one may ask themselves in which contexts it is possible to describe a discrete dynamical system on a continuum using a topologically conjugate symbolic dynamical systems. The answer is, in fact, never:
\begin{lemma}
\label{lem:1}Topological conjugacy between a symbolic dynamical system and a discrete dynamical system defined over a continuum is impossible.
\end{lemma}
\begin{proof}This follows immediately from the fact that all symbolic dynamical systems have an underlying metric space which is totally disconnected, that all continua are connected, and that it is impossible for a totally disconnected metric space to be homeomorphic to a continuum as the homeomorphism required by topological conjugation preserves connectivity.
\end{proof}
Given the statement in the lemma above, the traditional approach in the literature to demonstrate the existence of mixing/chaos within discrete dynamical systems on continua has been to either perform a direct proof of mixing or chaos in such a system, or to pivot on proving that such discrete dynamical systems contain a subset which is topologically conjugate to a mixing/chaotic symbolic dynamical system. Unfortunately, these subsets must necessarily encompass a trivial region of the continuum, a statement formalized by the notion of a punctiform space\cite{steen2013counterexamples}:
\begin{definition}
A metric space $(X,d)$ is punctiform if and only if there do not exist any subsets $W\in X$ such that the metric subspace $(W,d)$ would be a continuum.
\end{definition}
Punctiform spaces, as implied by the name, are point-like; there is no part of such a space that is a continuum. And, perhaps unsurprisingly, the underlying metric spaces associated with qualitative dynamical systems are punctiform spaces:
\begin{remark}
\label{rem:3}All totally disconnected metric spaces are punctiform.
\end{remark}
\begin{corollary}
    If $(Q,\ell,\sigma)$ is a qualitative dynamical system, $(Q,\ell)$ is a punctiform metric space.
\end{corollary}
\begin{corollary}
    If a subset $W \in X$ of a discrete dynamical system $(X,d,f)$ on a continuum exists such that $(W,d,f)$ is topologically conjugate to a qualitative dynamical system $(Q,\ell,\sigma)$, the subspace $(W,d)$ is punctiform.
\end{corollary}
Remark \ref{rem:3} is true thanks to the fact that the connected components of a totally disconnected space are singleton sets, and so there is no subset of a totally disconnected set which is both connected in the subspace topology and has more than one element, as required for it to be a continuum. Its corollaries follow from Remark \ref{rem:1} and the fact that total disconnectedness is preserved by the homemorphism implied by topological conjugacy, respectively.

This second corollary is the mathematically rigorous way of stating that, for any subset in a discrete dynamical system on a continuum, the existence of a topological conjugacy to a mixing/chaotic qualitative dynamical system only works to identify the presence of these properties in point-like regions within that continuum. As a result, applied dynamicists identifying such a conjugacy in a real system could face a scenario where the behavior in all nontrivially large regions of the continuum does not exhibit mixing or chaos, in spite of the conjugacy described above. In addition, those who instead attempt to be satisfied with establishing a semi-conjugacy with a qualitative dynamical system will find that there are trajectories in the qualitative dynamical system that describe trajectories that do not exist in the continuum system, as a result of the surjectiveness (and therefore lack of one-to-one correspondence) of the semi-conjugating map.

Therefore, it would be highly desirable to discover a workaround that would somehow allow an applied mathematician or scientist to fully describe the orbits of a discrete dynamical system over a continuum using a one-to-one correspondence to a qualitative dynamical system, and then use the existence (or absence) of mixing/chaotic behavior in the qualitative system to prove the existence of that behavior in the discrete system everywhere on the continuum. 

A curious approach to circumvent the limitation of Lemma \ref{lem:1} on the analysis of a discrete dynamical system $(X,d,f)$ where $(X,d)$ is a continuum, is to consider constructing an alternative metric $d^*$ for $X$ such that the metric topology of $(X,d^*)$ contains the metric topology of $(X,d)$, and then studying the dynamical properties of the auxiliary system $(X,d^*,f)$. One could then hope that this auxiliary system is topologically conjugate to a qualitative/symbolic dynamical system, and that the existence of mixing or chaos in the auxiliary qualitative system indicates its existence in the actual system through topological conjugacy and Remark \ref{rem:2}. 

The process of verifying the circumstances in which the above proposed analysis can work, or can't, must begin with an understanding of the metric space defined by $(X,d^*)$; a metric space termed in this text an "ultracontinuum" to reflect that its topology is finer than the topology of a continuum on the same underlying set.

\subsection{\label{Ultracontinuum}Ultracontinua}
\begin{definition}
 \label{def:11}   A metric space $(X,d^*)$ is an ultracontinuum if and only if there exists a metric $d$ such that $(X,d)$ is a continuum and that the metric topology of $(X,d^*)$ is strictly finer than the metric topology of $(X,d)$. 
\newline

    $(X,d^*)$ can then be referred to as an ultracontinuum of $(X,d)$.
\end{definition}
The main utility of the notion of an ultracontinuum in this text is, because the topology of the ultracontinuum contains every open set in the topology of its subsidiary continuum, that mixing behavior everywhere on the ultracontinuum implies it everywhere on the contained continuum.

\begin{lemma}
\label{lem:2}If $(X,d^*,f)$ is a topologically transitive/mixing discrete dynamical system such that $(X,d^*)$ is an ultracontinuum of $(X,d)$, then $(X,d,f)$ is correspondingly transitive/mixing.
\end{lemma}
\begin{proof}
    Either transitivity and mixing on the ultracontinuum system implies that their corresponding conditions must be satisfied for all pairs of open sets in the metric topology of $(X,d^*)$, $\tau^*$. Since the metric topology of $(X,d)$, $\tau$, is contained inside of $\tau^*$ by definition, these conditions must also be satisfied for any pair of open sets in $\tau$, indicating that $(X,d,f)$ is transitive/mixing if $(X,d^*,f)$ is transitive/mixing as well.
\end{proof}
An analogous property is also true for dense periodic orbits:
\begin{lemma}
\label{lem:3}If $(X,d^*,f)$ has dense periodic orbits and $(X,d^*)$ is an ultracontinuum of $(X,d)$, then $(X,d,f)$ has dense periodic orbits.
\end{lemma}
\begin{proof}
    Since every open set in $\tau^*$ must contains a periodic orbit, every open set in $\tau \subset \tau^*$ correspondingly contains a periodic orbit, and so $(X,d,f)$ also has dense periodic orbits.
\end{proof}
Under an additional requirement on $f$, we can state the same sort of result for chaoticity:
\begin{lemma}
\label{lem:4}
    If $(X,d^*,f)$ is chaotic, $(X,d^*)$ is an ultracontinuum of $(X,d)$, and $f$ is continuous on $(X,d)$, then $(X,d,f)$ is chaotic.
\end{lemma}
\begin{proof}
If $(X,d^*,f)$ is chaotic, then it is topologically transitive and contains dense periodic orbits; Lemmas \ref{lem:2} and \ref{lem:3} indicate that this must also be true for $(X,d,f)$. Because $f$ is also continuous on $(X,d)$, the Banks-Brooks-Cairns-Davis-Stacey theorem\cite{Banks1992} ensures that $(X,d,f)$ is also sensitive, and therefore chaotic.
\end{proof}
These three lemmas represent the key utility of the concept of ultracontinua to analyzing dynamical systems. 

Another convenient property of an ultracontinuum space is that an ultracontinuum transformed by a homeomorphism is still an ultracontinuum:
\begin{lemma}
\label{lem:5}
If two metric spaces $(X_1,d^*_1)$ and $(X_2,d^*_2)$ are homeomorphic and $(X_1,d_1^*)$ is an ultracontinuum, $(X_2,d^*_2)$ is an ultracontinuum.
\end{lemma}
\begin{proof}
    Assume $(X_1,d^*_1)$ is an ultracontinuum with metric topology $\tau_1^*$. By Definition \ref{def:11}, there must exist some metric $d_1$ and a metric topology $\tau_1$ it induces on $X_1$ such that $(X_1,d_1)$ is a continuum and $\tau_1 \subset \tau_1^*$. Therefore, the topological space $(X_1,\tau_1)$ is a compact, connected, metrizable topological space with uncountably many elements.

    Noting there must exist a homeomorphism $h$ between $(X_1,d^*_1)$ and $(X_2,d^*_2)$, and a related auxiliary bijection $h^\tau$ between $\tau_1^*$ and $\tau_2^*$, consider the restriction of $h^\tau$ to $\tau_1$ and its image $\tau_2 = h(\tau_1)$. This restriction is a bijection between $\tau_1$ and $\tau_2$, and so $(X_2,\tau_2)$ must be a compact, connected, metrizable topological space with uncountably many elements. This metrizability implies there exists a metric $d_2$ such that the metric space $(X_2,d_2)$ has the metric topology $\tau_2$.

    Since this metric space $(X_2,d_2)$ is compact, connected, and has uncountably many elements, and its metric topology $\tau_2 \subset \tau_2^*$ by construction, $(X_2,d_2)$ is a continuum and $(X_2,d_2^*)$ is an ultracontinuum of $(X_2,d_2)$.
\end{proof}

It is important to note that the topology of the ultracontinua does not necessarily have to be similar to that of the continuum whose topology it contains; in fact, we will positively (and perhaps counterintuitively) demonstrate that an ultracontinuum can be totally disconnected while having a necessarily connected subsidiary continuum. This is possible because the continuum is not actually a metric or topological subspace of the corresponding ultracontinuum; in fact, there need not be any obvious relationship between the metrics of the continuum and the ultracontinuum, enabling the ultracontinuum to have relatively exotic topologies.

However, there is one large restriction on the kind of topology an ultracontinuum can have:
\begin{lemma}
\label{lem:6}
There is no compact ultracontinuum.
\end{lemma}
\begin{proof}
    Consider an ultracontinuum $(X,d^*)$. By definition, there must exist some metric $d$ such that $(X,d)$ is a continuum, and such that the resulting metric topology of $(X,d)$, $\tau$, is strictly coarser than the metric topology $\tau^*$ of $(X,d^*)$; in short, $\tau \subset \tau^*$. 
    
    Now consider the identity mapping $i: (X,\tau^*) \to (X,\tau)$; this mapping is a continuous bijection, as every element of $\tau$ is also an element of $\tau^*$. 

    Suppose that $(X,d^*)$ were a compact metric space. If so, then $i$ would be a continuous bijection from a compact space to a metric space, which is automatically a homeomorphism\cite{Rudin1976-vd}. But this implies that $i$ is also a bijection between the topologies, and therefore $\tau^* = \tau$; violating the requirement that $\tau \subset \tau^*$. As a result, there is no ultracontinuum $(X,d^*)$ whose topology is compact.
\end{proof}
\begin{corollary}
There is no symbolic dynamical system which is topologically conjugate to a discrete dynamical system on an ultracontinuum.
\end{corollary}

This corollary appears to put a dent in the entire motivation for defining ultracontinuum spaces—the idea of establishing topological conjugacy to a symbolic dynamical system—but notice that this corollary implies nothing about the non-existence of an ultracontinuum which is topologically conjugate to a qualitative dynamical system. In the following section, we will demonstrate by example that it is possible to use this framework to establish a topological conjugacy between a qualitative dynamical system that is not a symbolic dynamical system and a discrete dynamical system on an ultracontinuum, and show that transitivity on the qualitative system implies transitivity for the ultracontinuum system on a subsidiary continuum.

\subsection{\label{sec:mixing}Example: Mixing in the Dyadic Map}

Consider the discrete dynamical system $(X,d,f)$ where $X = \left[0,1\right]$ is the closed unit interval, $d$ is the Euclidean metric on the interval such that 
\begin{eqnarray}
d(x,y) = ||x - y||^2 \quad \forall x, y \in X
\end{eqnarray}
and $f: [0,1] \to [0,1]$ is the dyadic map, defined such that
\begin{eqnarray}
f(x) = 
    \begin{cases}
        2x & \text{if } 0 \leq x < \frac{1}{2}\\
        2x - 1 & \text{if } \frac{1}{2} \leq x \leq 1\\
    \end{cases}
\end{eqnarray}
For brevity, we can refer to the metric space $(X,d)$, often called the "standard unit interval", with the symbol $\mathbf{I}$ where convenient. A graph of the dyadic map, as defined above, is shown in Fig.~\ref{fig:dyadicmap}---notice that the mapping is discontinuous on $\mathbf{I}$, and therefore lies outside of the realm of applicability of Lemma \ref{lem:4}. The value of the map at the discontinuity selected in Eq. 2 is not required for the general validity of the following analytical technique; it can be straightforwardly altered to accommodate selecting $f(\frac{1}{2}) = 1$ instead of $0$.
\begin{figure}
\includegraphics[width=0.47\textwidth,keepaspectratio]{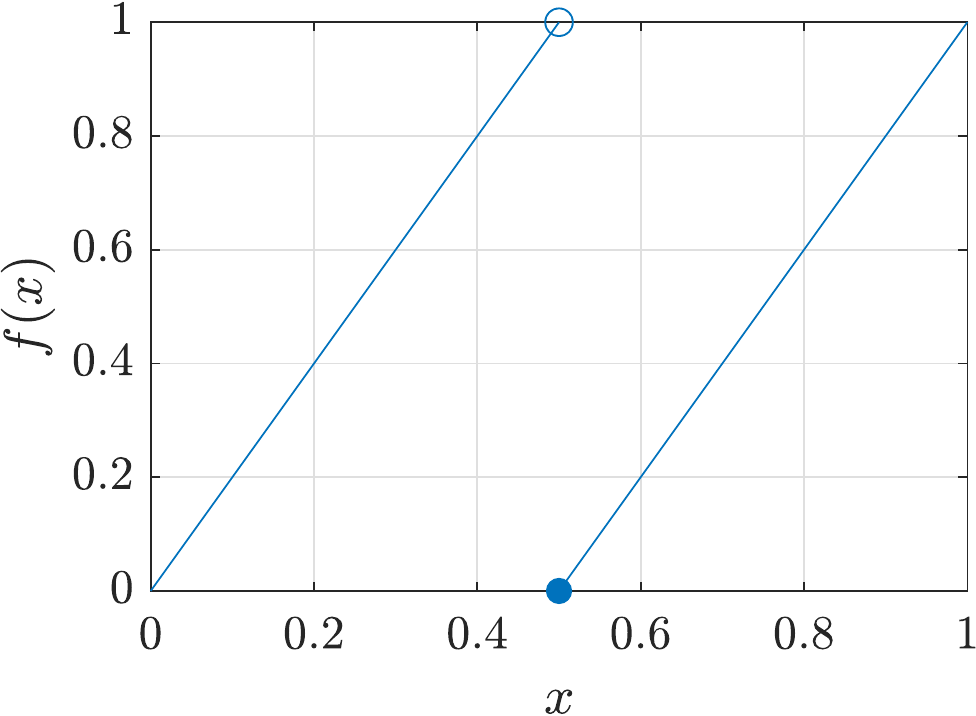}
\caption{\label{fig:dyadicmap} A graph of the dyadic map $f$, as defined in Equation 2. The dot indicates the value the function takes in the discontinuity at $x = \frac{1}{2}$.}
\end{figure}

In addition to being considered on its own as a prototypical example of a chaotic map in texts on nonlinear dynamics and chaos\cite{alligood2012chaos}, the dyadic map has been shown to be topologically semi-conjugate with the tent map of unit height and the logistic map with the parameter value $4$ when both act on the unit interval\cite{alligood2012chaos,robinson1998dynamical}; both of which are also key pedagogical examples in nonlinear dynamics and chaos\cite{robinson1998dynamical,strogatz2014nonlinear,alligood2012chaos}.

One can describe the dynamics of an element in $[0,1]$ under the dyadic map straightforwardly by considering a binary expansion of that element. By picking the expansion for every element in $[0,1)$ that contains no infinite trailing $1$'s in the numerator (and such that $1$'s binary expansion is picked to be the one containing exclusively $1$'s in the numerator), the dynamics of the dyadic map can be described by simply shifting the numerators of every term in the expansion "forwards" and "deleting" what was previously the initial numerator in the expansion.

To illustrate, consider $\frac{3}{4}$. Its binary expansion per the above scheme is:
\begin{eqnarray}
    \frac{3}{4} = \frac{1}{2} + \frac{1}{2^2} + \frac{0}{2^3} + \frac{0}{2^4} + \frac{0}{2^5} + \dots
\end{eqnarray}
Evolving this element with the dyadic map results in the following element:
\begin{eqnarray}
    f\left(\frac{3}{4}\right) =  \frac{1}{2} = \frac{1}{2} + \frac{0}{2^2} + \frac{0}{2^3} + \frac{0}{2^4} + \frac{0}{2^5} + \dots
\end{eqnarray}
Finally, evolving the resulting $\frac{1}{2}$ generates the expected binary expansion:
\begin{eqnarray}
    f^2\left(\frac{3}{4}\right) =  0 = \frac{0}{2} + \frac{0}{2^2} + \frac{0}{2^3} + \frac{0}{2^4} + \frac{0}{2^5} + \dots
\end{eqnarray}
All future iterations will result in the same element $0$, and by extension the same binary expansion (as predicted by the claim above).

This description of the elements of $[0,1]$, and how the effect of the dyadic map can be described with them, lends itself to a natural qualitative description of the dynamics of the dyadic map on the unit interval. The traditional approach is to partition the interval into "left" ($0\leq x<\frac{1}{2}$) and "right" portions ($\frac{1}{2}\leq x\leq1$), and to describe the dynamics of an element $x \in [0,1]$ by assigning it an infinite sequence of symbols from the alphabet $\Sigma = \{L,R\}$, where the $n$th symbol in the sequence indicates if the element is on the left ($L$) or right ($R$) side of the partition after $n-1$ applications of the dyadic map. Conveniently, the sequence of $L$'s and $R$'s assigned to an element of $[0,1]$ through this process is precisely the sequence of numerators in the binary expansion of that number after replacing $L$'s with $0$'s and $R$'s with $1$'s. 

As an example, the element $\frac{1}{2}$ would be assigned the infinite sequence $RLLLLL\dots$, as it is originally located on the right side of the partition as defined above, but then evolves into $0$ on the left side and remains there for all future iterations. Replacing the $R$'s and $L$'s with $1$'s and $0$'s, this is precisely the sequence of numbers in the numerators of the binary expansion of $\frac{1}{2}$ shown in Equation 4.

The set of all such possible infinite sequences, $Q$, equipped with the shift map $\sigma$ and an appropriate metric $\ell$ that induces the required metric topology for $(Q,\ell,\sigma)$ to be a qualitative dynamical system according to Definition \ref{def:1}, is the qualitative dynamical system of interest with which we will try to demonstrate topological conjugacy to the dyadic map on the interval. Although the precise choice of metric $\ell$ is not important as long as it generates the correct topology, we can consider a specific choice for clarity that is used commonly in symbolic dynamics texts\cite{lind_marcus_1995, Kitchens1998, Block2006-re}:
\begin{eqnarray}
    \ell(x,y) =  \begin{cases}
        2^{-k} & \text{if } x \neq y\\
        0 & \text{if } x = y
    \end{cases}
\end{eqnarray}
where $k$ is the length of the largest prefix of symbols that $x$ and $y$ share.
It therefore makes sense to identify certain interesting properties of this set $Q$, as well as the metric space $(Q,\ell)$ it generates.

The first thing to note is that not all infinite sequences made from $L$'s and $R$'s describe the orbit of an element in the unit interval under the dyadic map. For example, there is no element in $[0,1]$ associated with the infinite sequence $LRRRRRR\dots$, as there is no element in $[0,1]$ smaller than $\frac{1}{2}$ that maps to $1$ through the dyadic map. In fact, there is no element other than $1$ that can be described by an infinite sequence of $L$'s and $R$'s that terminates with an infinite sequence of $R$'s; the present or future value of any such element must contain a "suffix" of the form $LRRRRR\dots$, and therefore at some point must be less than $\frac{1}{2}$ but map to $1$, which was previously described as impossible. 

$Q$ can therefore be fully described as the union of the set of all infinite symbol sequences made from the alphabet $\Sigma = \{L,R\}$ that do not contain a terminating infinite string of $R$'s, and the single infinite sequence consisting entirely of all $R$'s. These "missing" trajectories in $Q$ generate a perhaps surprising topological consequence for the metric space $(Q,\ell)$:
\begin{lemma}
\label{lem:7}    $(Q,\ell)$ is not a compact metric space.
\end{lemma}
\begin{proof}
    A sufficient condition for $(Q,\ell)$ to not be compact is if there exists an infinite open cover for $Q$ such that no open set in the cover can be removed from it without causing the remaining collection of open sets to not cover $Q$. As a convenient shorthand for describing certain open sets in this space, we can use a notation such that a finite symbol sequence of $L$'s and $R$'s followed by an underscore indicates the set of all elements in $Q$ that share that finite symbol sequence as a prefix; for example, $LR\_$ represents the set of all elements in $Q$ that share the prefix $LR$. These sets are, by Definition \ref{def:1}, always open sets in the metric topology of $(Q,\ell)$.

    Now consider the cover of $Q$ generated by the open sets $R\_$, $LL\_$, and the following infinite sequence of open sets:
    \begin{eqnarray}
        LRL\_,\ LRRL\_,\ LRRRL\_,\ LRRRRL\_,\ LRRRRRL\_,\ \dots
    \end{eqnarray}
    This is indeed a cover of $Q$ consisting of an infinite number of open sets, but none of them can be removed without causing the collection to fail to be a cover of $Q$. Therefore, $(Q,\ell)$ cannot be compact.
\end{proof}
\begin{corollary}
$(Q,\ell,\sigma)$ is a qualitative dynamical system, but not a symbolic dynamical system.
\end{corollary}
Since $(Q,\ell)$ is not compact, it cannot be homeomorphic to the unit interval $[0,1]$ equipped with the Euclidean topology $d$ (which is compact), and by extension $(Q,\ell,\sigma)$ cannot be topologically conjugate to $([0,1],d,f)$. 

However, we can still use the correspondence described previously between elements of $Q$ and the binary expansion of numbers in $[0,1]$ to construct a bijection $h: Q \to [0,1]$, identifying every number in the unit interval with a qualitative description of its dynamics under the dyadic map. We can then convert $h$ into an isometry, and by extension a homeomorphism, by defining a metric $d^*$ on $[0,1]$ such that 
\begin{eqnarray}
    d^*(x,y) = \ell\left(h^{-1}(x),h^{-1}(y)\right) \quad \forall x, y \in [0,1]
\end{eqnarray}
In other words, the distance between two numbers in $[0,1]$ as defined by $d^*$ is the distance between their qualitative descriptions in $Q$ as defined by $\ell$.
\begin{remark}
    \label{rem:4}$(Q,\ell,\sigma)$ is topologically conjugate to $([0,1],d^*,f)$.
\end{remark}
\begin{proof}
The isometry $h$ described previously is automatically a homeomorphism, and the effect of the shift map on an infinite sequence $x\in Q$ of $L$'s and $R$'s matches the effect of the dyadic map on the sequence of numerators in the binary expansion of the number $h(x) \in [0,1]$. As a result, $h\circ \sigma (x) =  f \circ h(x)$ for all $x \in Q$, as required for topological conjugacy.
\end{proof}

Like we did for the unit interval equipped with the Euclidean metric, we can refer to the metric space $([0,1],d^*)$ using the symbol $\mathfrak{I}_1$. And since $h$ is a homeomorphism, $\mathfrak{I}_1$ inherits the following topological properties from $(Q,\ell)$:
\begin{remark}
\label{rem:5}
$\mathfrak{I}_1 = \left([0,1],d^*\right)$ is a totally disconnected, non-compact metric space.
\end{remark}
Crucially, it also has the following important property:
\begin{theorem}
    $\mathfrak{I}_1$ is an ultracontinuum of $\mathbf{I}$.
    \label{thm:2}
\end{theorem}
\begin{proof}
The main strategy is to demonstrate that any open ball in $\mathbf{I}$ can be described as the union of the images of (potentially infinitely many) cylinder sets in $(Q,\ell)$ under the application of $h$. Since such a union must also be an open set in $\mathfrak{I}_1$,  and the open balls of $\mathbf{I}$ form a basis for the topology of $\mathbf{I}$, this would immediately imply that every open set in the metric topology of $\mathbf{I}$ is contained in the metric topology of $\mathfrak{I}_1$ as required. For simplicity, the notation for open sets of $Q$ established in the proof of Lemma \ref{lem:7} shall be reused, in which (for example) $RL\_$ refers to the set of all elements of $Q$ that share the prefix $RL$.

Consider the image of a cylinder set $P \in Q$, consisting of all elements in $Q$ that share some prefix $p$, under the action of $h$. Using the standard ordering of elements on the unit interval, the smallest element in $h(P)$ corresponds to $a = h\left(p^\frown LLLL\dots\right)$, where $\frown$ represents the concatenation operator. $a$ is, by the construction of $h$, a number in $[0,1]$ whose binary expansion terminates, and is therefore a dyadic rational—a rational number that can be expressed as a fraction whose denominator is a power of $2$. 

Using similar logic, $h(P)$ must contain the image of every element in $Q$ of the form $p^\frown R^{m\frown} LLL\dots$, where $R^m$ indicates $m \in \mathbb{N}$ repetitions of the symbol $R$ in the sequence. As a result, there exists a supremum $b$ for $h(P)$ in $[0,1]$ that is either $1$ (if $p$ consists of all $R$'s) or is not within $h(P)$ (if $p$ contains at least one $L$ in its sequence). In the latter case, this supremum $b$ is the image of the element of $Q$ with prefix $p_+$, created by switching the last $L$ symbol present in $p$ to an $R$, concatenated with an infinite $LLL\dots$ sequence. In either case, $b$ is also a dyadic rational.

Since $h(P)$ must contain every element in $[0,1]$ between $a$ and $b$ by the properties of $h$ and the real numbers, and since the above arguments are true for any cylinder set $P$ in the topology of $Q$, $\mathfrak{I}_1$ has a topological basis $\mathcal{B}$ consisting of half-open intervals of the form $[a,b)$ and $[a,1]$ for dyadic rationals $a$ and $b$ in the unit interval. 

Similarly, for any dyadic rationals $a, b \in [0,1]$, the half-open interval $[a,b)$ and the closed interval $[a,1]$ is an open set of $\mathfrak{I}_1$. In the former case, any such set can be constructed explicitly from the basis $\mathcal{B}$ by considering the sequence of $0$'s and $1$'s in the binary expansions of $a$ and $b$; if the length of the largest initial sequence of $0$'s and $1$'s in both $a$ and $b$'s binary expansion that ends with a $1$ is $t$, then one can construct two length-$t$ sequence of $L$'s and $R$'s, $p_a$ and $p_b$, by taking the first $t$ numerators of $a$ and $b$'s binary expansions and convert them to finite $L$-$R$ sequences by assigning $0\to L$ and $1 \to R$. Defining $P_a$ as the set of all elements in $Q$ that share the prefix $p_a$, $[a,b)$ can be constructed as the following union of open sets in $\mathcal{B}$:
\begin{eqnarray}
    [a,b) = h(P_a)\cup \left(\bigcup_i h\left(P_i\right)\right)
\end{eqnarray}
where each cylinder set $P_i$ consists of every element of $Q$ that shares a length-$t$ prefix $p_i$ located between $p_a$ and $p_b$ in standard lexicographical/dictionary order for $L<R$. In the case where one is interested in constructing $[a,1]$, one can use the same process as described above, but instead identifying $t$ as the length of the largest initial sequence of $0$s and $1$s in $a$'s binary representation that terminates in $1$, and replacing $p_b$ with the length-$t$ sequence of $R$'s.

Since the dyadic rationals are dense in $\mathbf{I}$, for any real number $z\in [0,1]$, one can construct sequences of dyadic rationals $z^-_n$ and $z^+_n$ that converge to $z$ such that $z^-_n < z$ for all $n\in\mathbb{N}$ and $z^+_n > z$ for all $n\in\mathbb{N}$. 

As a result, one can construct any open ball in the topology of $\mathbf{I}$ through the following operations on open sets in $\mathfrak{I}_1$:
\begin{subequations}
\begin{eqnarray}
[0,x_1) = \bigcup_{n=1}^{\infty} [0,z^-_n)
\end{eqnarray}
\begin{eqnarray}
(x_2,1] = \bigcup_{n=1}^{\infty} [z^+_n,1]
\end{eqnarray}
\begin{eqnarray}
(x_2,x_1) = [0,x_1) \cap (x_2,1]
\end{eqnarray}
\end{subequations}
for any $x_1,\ x_2 < x_1 \in [0,1]$ and any appropriate $z^-_n$ and $z^+_n$ sequences.

Therefore, every open ball in $\mathbf{I}$ can be constructed out of countable unions and finite intersections of open sets of $\mathfrak{I}_1$. Since every open set in the metric topology of $\mathbf{I}$ can be represented as a union of these open balls, and because these set operations imply that such open balls are open sets of $\mathfrak{I}_1$, every open set in $\mathbf{I}$ is an open set of $\mathfrak{I}_1$. 

Consequently, $\mathfrak{I}_1$ is an ultracontinuum of $\mathbf{I}$.
\end{proof}

Notice that the metric $d^*$ on numbers in the unit interval behaves very differently than its Euclidean counterpart $d$; see Figure \ref{fig:ultrametric} for a comparison of their values using the reference element $\frac{1}{2}$.

\begin{figure}
\includegraphics[width=0.47\textwidth,keepaspectratio]{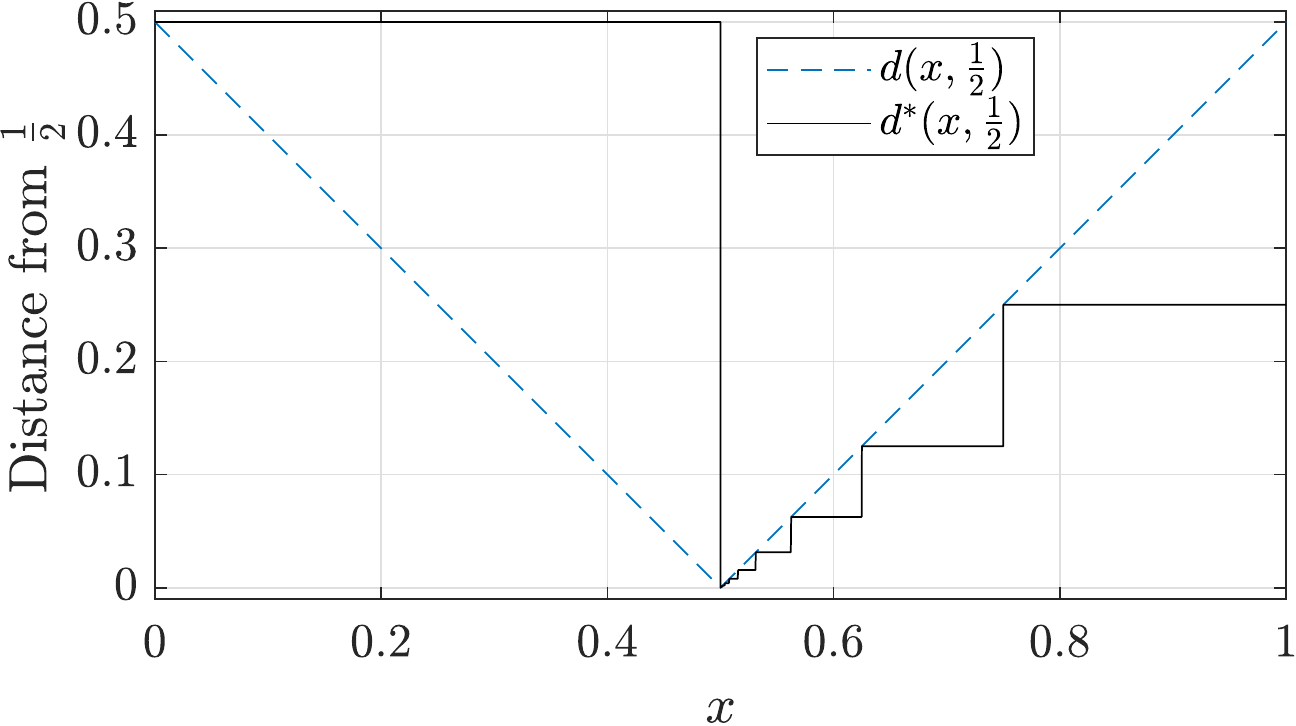}
\caption{\label{fig:ultrametric} A graph of the distance between a given element $x \in [0,1]$ and $\frac{1}{2}$, as "measured" by the standard Euclidean metric $d$ and the ultracontinuum metric $d^*$ described previously. Note the self-similar structure of the graph for the ultracontinuum metric as $x\to \frac{1}{2}^+$.}
\end{figure}

As a result of these differences, it is difficult to establish direct comparisons of distances between elements in either metric space. For example, consider the following two sequences of distances between identical elements in $[0,1]$:
\begin{subequations}
\label{eq:whole}
\begin{eqnarray}
s^*_n = d^*\left(\frac{1}{2},\frac{1}{2} - \frac{1}{2^{(n+1)}}\right)
\end{eqnarray}
\begin{eqnarray}
s_n = d\left(\frac{1}{2},\frac{1}{2} - \frac{1}{2^{(n+1)}}\right)
\end{eqnarray}
\end{subequations}
for $n \in \mathbb{N}$. $s_n$ clearly converges to $0$ as $n \to \infty$, but $s^*_n = \frac{1}{2}$ for all $n \in \mathbb{N}$, demonstrating that statements about separation in the ultracontinuum may not hold in the subsidiary continuum; an obstacle towards making direct claims, for example, about sensitivity to initial conditions that hold in both spaces.

These results set the stage for the analysis to follow; if $(Q, \ell, \sigma)$ can be shown to be topologically mixing, topological conjugacy ensures that $([0,1],d^*,f)$ is also mixing (Remark \ref{rem:2}), and Lemma \ref{lem:2} with Theorem \ref{thm:2} guarantee $([0,1],d,f)$ is topologically mixing as well. 
\begin{theorem}
    $(Q,\ell,\sigma)$ is topologically mixing.
\label{thm:3}
\end{theorem}
\begin{proof}
    Given that the cylinder sets of $(Q,\ell)$ form a topological basis, it suffices to show that for every pair of cylinder sets $U$ and $V$ in the topology of $(Q,\ell)$, there exists an integer $N\geq 0$ such that $\sigma^n(U) \cap V \neq \emptyset$ for all integer $n \geq N$; since every open set in the topology of this space can be represented as a union of these cylinder sets, topological mixing on these sets implies mixing on the whole space. Once again, the shorthand notation for these sets used in the proof of Lemma \ref{lem:7} is used (where, for example, $LR\_$ is the set of all elements of $Q$ that share the prefix $LR$).

    Consider any two such cylinder sets $p_1\_$ and $p_2\_$, where $p_1$ and $p_2$ are some finite sequence of $L$'s and $R$'s of length $m_1$ and $m_2$ respectively. For any integer $n \geq m_1$, we can identify an element $x \in Q$ of the form 
    \begin{equation}
        x = p_1^\frown L^{(n - m_1)\frown} p_2 ^\frown LLLLLLL\dots
    \end{equation}
    where $L^{n - m_1}$ is a sequence of $n - m_1$ consecutive $L$'s. This element clearly belongs to $p_1\_$ and, after $n$ shifts, belongs both to $\sigma^n(p_1\_)$ and $p_2\_$, indicating the intersection of these open sets is non-empty.

    Since such an element $x$ can be constructed for all $n \geq m_1$, then $\sigma^n(p_1\_)\cap p_2\_ \neq \emptyset$ for all $n \geq m_1$, which immediately implies that there exists an $N \in \mathbb{N}$, $m_1$, for any two cylinder sets $p_1\_$ and $p_2\_$ such that $\sigma^n(p_1\_)\cap p_2\_ \neq \emptyset$ for all $n \geq N$. Finally, since every open set in the topology of $(Q,\ell)$ is a union of these cylinder sets, there exists an $N \in \mathbb{N}$ for any two open sets $U$ and $V$ such that $\sigma^n(U)\cap V \neq \emptyset$ for all $n \geq N$, demonstrating that $(Q,\ell,\sigma)$ is topologically mixing.
\end{proof}
\begin{corollary}
    The dyadic map on the unit interval, equipped with the Euclidean metric, is topologically mixing.
\end{corollary}
\begin{proof}
    Since $(Q,\ell,\sigma)$ is topologically mixing, and $([0,1],d^*,f)$ is topologically conjugate to it (Remark \ref{rem:4}), Remark \ref{rem:2} implies $([0,1],d^*,f)$ is also topologically mixing. And since $\mathfrak{I}_1$ is an ultracontinuum of $\mathbf{I}$ (Theorem \ref{thm:2}), $([0,1],d,f)$ is topologically mixing as well (Lemma \ref{lem:2}).
\end{proof}
\subsection{\label{chaos}Example: Chaos in the 2-Tent Map}
Consider the discrete dynamical system $(X,d,f_2)$ where $X = \left[0,1\right]$ is the closed unit interval, $d$ is the Euclidean metric on the interval, and $f_2: [0,1] \to [0,1]$ is the 2-tent map, defined such that
\begin{eqnarray}
f_2(x) = 
    \begin{cases}
        2x & \text{if } 0 \leq x < \frac{1}{2}\\
        2(1-x) & \text{if } \frac{1}{2} \leq x \leq 1\\
    \end{cases}
\end{eqnarray}
An image of this mapping is shown in Figure \ref{fig:tentmap2}.

\begin{figure}
\includegraphics[width=0.47\textwidth,keepaspectratio]{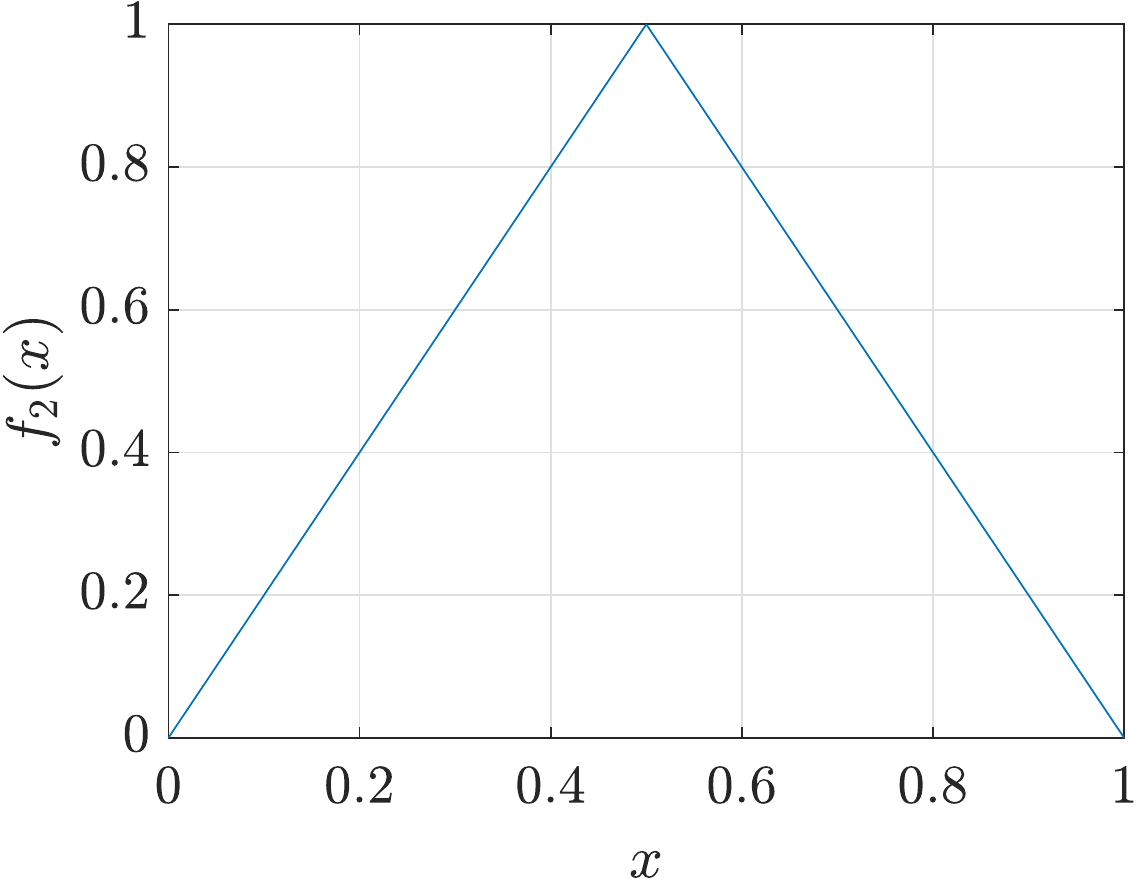}
\caption{\label{fig:tentmap2} A graph of the $2$-tent map $f_2$ described in Equation 12. This mapping is continuous, unlike the dyadic map shown in Fig. \ref{fig:dyadicmap}, and maps from $[0,1]$ to itself unlike the $3$-tent map in Fig. \ref{fig:tentmap3}.}
\end{figure}
This map is quite similar to the dyadic map from the previous section; it is also defined on the "standard" unit interval $\mathbf{I} = ([0,1],d)$, and is also piecewise linear. The key difference is that the tent map is continuous on $\mathbf{I}$, and therefore the Banks-Brooks-Cairns-Davis-Stacey theorem applies to it; the goal is to demonstrate that this mapping is chaotic using the tools described and used previously, along with this theorem.

Tent maps in general---those in which the $2$ in the definition of $f_2$ is replaced with some other real number---are considered the elementary paradigm for so-called "stretch-and-fold" chaotic behavior and appear frequently in introductory discussions on the subject\cite{alligood2012chaos,Elaydi2007}. As such, this specific form of the tent map can also be considered the elementary paradigm for stretch-and-fold chaos on a continuum.

If we qualitatively partition the interval into the same regions than those generated in Section \ref{sec:mixing}---$L$ representing $\left[0,\frac{1}{2}\right)$ and $R$ representing $\left[\frac{1}{2},1\right]$---we can attempt to construct a conjugacy between $(X,d,f_2)$ and some qualitative dynamical system $(Q_2,\ell,\sigma)$ where $Q_2$ is made from the alphabet $\Sigma = \{L,R\}$.

As before, not all infinite symbol sequences of $L$'s and $R$'s represent a trajectory in $(X,d,f_2)$. For example, there is no element in $\mathbf{I}$ whose dynamics is described by the infinite symbol sequence $LRLLLLL\dots$; as the only number that maps to $0$ on the right side of the interval is $1$, and there is no number smaller than $\frac{1}{2}$ that maps to $1$. By extension, no infinite symbol sequence which contains the "suffix" $LRLLLLLL\dots$ can belong in $Q_2$.

Although the dynamics of the tent map are more complex than that of the dyadic map, one can readily develop a bijection between $[0,1]$ and the elements of $Q_2$ via the same process in the previous example; creating a correspondence between the infinite sequences in $Q_2$, and the infinite sequences $\chi_i$ of numerators in the binary expansion of a number in $[0,1]$. Implementing the $LRLLLL\dots$ restriction described previously, one can use the Mealy machine\cite{Mealy1955} in Fig. \ref{fig:mealy} to transform a qualitative sequence in $Q_2$ into a number in $[0,1]$ whose dynamics under the tent map matches that sequence. This process matches the algorithmic process described by Cvitanovi{\'c}, Artuso, Mainieri, Tanner \& Vattay\cite{ChaosBook}, and implicitly defines a bijection $h_2: Q_2 \to [0,1]$.
\begin{figure}
\includegraphics[width=0.47\textwidth,keepaspectratio]{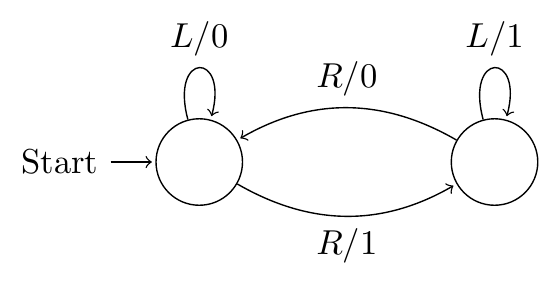}
\caption{\label{fig:mealy} The following automaton converts infinite sequences of $L$'s and $R$'s into $\chi_i$ sequences of $0$'s and $1$'s. At the "starting" node, the automaton reads the first symbol of the $L$-$R$ sequence, outputs the symbol ($0$ or $1$) on the corresponding outgoing arrow from that node, and then transitions to the node to which the arrow is pointing to. The process thus repeats for the current node and following $L$-$R$ symbol. These $\chi_i$ sequences represent the numerators of the binary expansion of a number in $[0,1]$, which by extension associates a qualitative sequence in $Q_2$ with the number in $[0,1]$ whose dynamics under the $2$-tent map are described by that qualitative sequence.}
\end{figure}

As an example, consider the period-2 sequence $LRLRLRLR\dots \in Q_2$. According to the Mealy machine in Fig. 5, this generates the period-4 $\chi_i$ sequence $01100110\dots$, which then corresponds to the number $2/5$:
\begin{eqnarray}
    \frac{2}{5} = \frac{0}{2} + \frac{1}{2^2} + \frac{1}{2^3} + \frac{0}{2^4} + \frac{0}{2^5} + \frac{1}{2^6} + \frac{1}{2^7} + \frac{0}{2^8}\dots
\end{eqnarray}
One can readily verify that the dynamics of $2/5$ are indeed well-described by the originating qualitative sequence $LRLRLR\dots$:
\begin{subequations}
\label{eq:210evol}
\begin{eqnarray}
    \frac{2}{5} \in \left[0,\frac{1}{2}\right)
\end{eqnarray}
\begin{eqnarray}
    f_2\left(\frac{2}{5}\right) = \frac{4}{5}  \in \left[\frac{1}{2},1\right]
\end{eqnarray}
\begin{eqnarray}
    f^2_2\left(\frac{2}{5}\right) = \frac{2}{5} \in \left[0,\frac{1}{2}\right)
\end{eqnarray}
\end{subequations}
By defining a metric $d_2^*$ similar to the metric $d^*$ defined in Eq. 8,
\begin{eqnarray}
    d_2^*(x,y) = \ell\left(h_2^{-1}(x),h_2^{-1}(y)\right) \quad \forall x, y \in [0,1]
\end{eqnarray}
$h_2$ is an isometry and therefore a homeomorphism from $(Q_2,\ell)$ to $([0,1],d_2^*)$.
\begin{remark}
\label{rem:6}
    $(Q_2,\ell,\sigma)$ is topologically conjugate to $(X,d^*_2,f_2)$.
\end{remark}
\begin{proof}
    The isometry $h_2$ described previously is automatically a homeomorphism, and the effect of the shift map on an infinite sequence $x\in Q_2$ of $L$'s and $R$'s matches the effect of the $2$-tent map on the $\chi_i$ sequence of numerators in the binary expansion of the number $h^{-1}_2(x) \in [0,1]$. As a result, $h_2\circ \sigma (x) =  f_2 \circ h_2(x)$ for all $x \in Q_2$, as required for topological conjugacy.
\end{proof}
The process that remains is quite similar to that in Section \ref{sec:mixing}; establish that $\mathfrak{I}_2 = (X,d_2^*)$ is an ultracontinuum of $\mathbf{I}$, and then show that $(Q_2,\ell,\sigma)$ is chaotic, which then implies that $(X,d,f_2)$ is chaotic as well. 
\begin{theorem}
\label{thm:4}
    $\mathfrak{I}_2$ is an ultracontinuum of $\mathbf{I}$.
\end{theorem}
\begin{proof}
The strategy is---as in the proof of Theorem \ref{thm:2}---to demonstrate that any open ball in $\mathbf{I}$ can be described as the union of the images of (potentially infinitely many) cylinder sets in $(Q_2,\ell)$. The notation for open sets of $Q_2$ established in the proof of Lemma \ref{lem:7} and elsewhere is reused, such that $lRL\_$ refers to the set of all elements of $Q_2$ that share the prefix $LRL$.

It is helpful for intuition to observe how the image of all cylinder sets with prefixes of specific length partition the interval. For example, $h_2(L\_) = \left[0,\frac{1}{2}\right)$ and $h_2(R\_) = \left[\frac{1}{2},1\right]$ by definition. Now consider the cylinder sets $LL\_$, $LR\_$, $RL\_$, and $RR\_$---their images onto the interval under the action of $h_2$ can be readily verified through $h_2$ and/or the Mealy machine in Fig. \ref{fig:mealy} to be:
\begin{subequations}
\label{eq:partitions1}
\begin{eqnarray}
h_2(LL\_) = \left[0,\frac{1}{4}\right)
\end{eqnarray}
\begin{eqnarray}
h_2(LR\_) = \left[\frac{1}{4},\frac{1}{2}\right)
\end{eqnarray}
\begin{eqnarray}
h_2(RL\_) = \left(\frac{3}{4},1\right]
\end{eqnarray}
\begin{eqnarray}
h_2(RR\_) = \left[\frac{1}{2},\frac{3}{4}\right]
\end{eqnarray}
\end{subequations}
Now consider the "subpartition" of $h_2(RL\_)$ into $h_2(RLL\_)$ and $h_2(RLR\_)$:
\begin{subequations}
\label{eq:partitions2}
\begin{eqnarray}
h_2(RLL\_) = \left(\frac{7}{8},1\right]
\end{eqnarray}
\begin{eqnarray}
h_2(RLR\_) = \left(\frac{3}{4},\frac{7}{8}\right]
\end{eqnarray}
\end{subequations}
As partially illustrated by Eqs. \ref{eq:partitions1} \& \ref{eq:partitions2}, the endpoints and type of interval that each of these cylinder sets subpartitions into follow self-similar rules. For example, any image of a cylinder set of the form $\left[a,b\right)$ can be partitioned into two cylinder sets $\left[a,\frac{b}{2}\right)$ and $\left[\frac{b}{2},b\right)$ for $a, b \in [0,1]$. Similarly, such a set of the form $\left(a,b\right]$ can be split into two cylinder sets $\left(a,\frac{b}{2}\right]$ and $\left(\frac{b}{2},b\right]$, and sets of the form $\left[a,b\right]$ can be split into $\left[a,\frac{b}{2}\right]$ and $\left(\frac{b}{2},b\right]$. From induction, each of these endpoints must be dyadic rationals.

Given any dyadic rationals $a,b \in [0,1]$, one can therefore always construct open sets in $\mathfrak{I}_2$ of the form $[0,a)$ for $a \leq 1/2$, $[0,a]$ for $a > 1/2$, $(b,1]$ for $b \geq 3/4$ and $[b,1]$ for $b < 3/4$ by judiciously subpartitioning, and taking unions of, the images of cylinder sets described previously. 

Since the dyadic rationals are dense in $\mathbf{I}$, for any number $x\in [0,1]$, one can construct sequences of dyadic rationals $z^-_n$ and $z^+_n$ that converge to $x$ such that $z^-_n < z$ for all $n\in\mathbb{N}$ and $z^+_n > z$ for all $n\in\mathbb{N}$. 

As a result, one can construct any open ball in the metric topology of $\mathbf{I}$ through the following operations on open sets in $\mathfrak{I}_2$:

\begin{subequations}
\begin{eqnarray}
[0,x_1) = \bigcup_{n=1}^{\infty} [0,z^-_n)\ \text{if} \ x_1 \leq \frac{1}{2}, \ \text{else} \ \bigcup_{n=1}^{\infty} [0,z^-_n]
\end{eqnarray}
\begin{eqnarray}
(x_2,1] = \bigcup_{n=1}^{\infty} (z^+_n,1]\ \text{if} \ x_2 \geq \frac{3}{4}, \ \text{else} \ \bigcup_{n=1}^{\infty} [z^+_n,1]
\end{eqnarray}
\begin{eqnarray}
(x_2,x_1) = [0,x_1) \cap (x_2,1]
\end{eqnarray}
\end{subequations}
for any $x_1,\ x_2 < x_1 \in [0,1]$ and any appropriate $z^-_n$ and $z^+_n$ sequences.

Therefore, every open ball in $\mathbf{I}$ can be constructed out of countable unions and finite intersections of open sets of $\mathfrak{I}_2$. Since every open set in the metric topology of $\mathbf{I}$ can be represented as a union of these open balls, and because these set operations imply that such open balls are open sets of $\mathfrak{I}_2$, every open set in $\mathbf{I}$ is an open set of $\mathfrak{I}_2$. 

Consequently, $\mathfrak{I}_2$ is an ultracontinuum of $\mathbf{I}$.
\end{proof}
\begin{corollary}
    $\mathfrak{I}_2$ and $(Q_2,\ell)$ are not compact.
\end{corollary}
All that's left is to prove that the qualitative dynamical system that reflects the dynamics of the tent map on the interval is chaotic---that the tent map is itself chaotic on $\mathbf{I}$ follows via corollary.
\begin{theorem}
\label{thm:5}
    $(Q_2,\ell,\sigma)$ is chaotic.
\end{theorem}
\begin{proof}
    Given that the cylinder sets of $(Q_2,\ell)$ form a topological basis, it suffices to show that for every pair of cylinder sets $U$ and $V$ in the topology of $(Q_2,\ell)$, there exists an integer $n\geq 0$ such that $\sigma^n(U) \cap V \neq \emptyset$; since every open set in the topology of this space can be represented as a union of these cylinder sets, topological transitivity on these sets implies mixing on the whole space. Once again, the shorthand notation for these sets used in the proof of Lemma \ref{lem:7} is used (where, for example, $LR\_$ is the set of all elements of $Q$ that share the prefix $LR$).

    Consider any two such cylinder sets $p_1\_$ and $p_2\_$, where $p_1$ and $p_2$ are some finite sequence of $L$'s and $R$'s of length $m_1$ and $m_2$ respectively. For any such cylinder sets, we can identify an element $x \in Q$ of the form 
    \begin{equation}
        x = p_1^\frown p_2 ^\frown RRRRRR\dots
    \end{equation}
     This element clearly belongs to $p_1\_$ and, after $m_1$ shifts, belongs both to $\sigma^n(p_1\_)$ and $p_2\_$, indicating the intersection of these open sets is non-empty.

    Since such an element $x$ can be constructed for all choices $p_1\_$ and $p_2\_$, then $\sigma^n(p_1\_)\cap p_2\_ \neq \emptyset$ for all $n \geq m_1$, and since every open set in the topology of $(Q_2,\ell)$ is a union of these cylinder sets, $(Q_2,\ell,\sigma)$ is topologically transitive.

    To demonstrate $(Q_2,\ell,\sigma)$ has dense periodic orbits, consider any cylinder set $p_3\_$ in the topology of $(Q_2,\ell)$.  For any such set, there exists an element $ y \in Q_2$ and in $p_3\_$ of the form $p_3^{\frown}p_3^{\frown}p_3^{\frown}\dots$, which is clearly periodic under the action of the shift map. Since such an element can be constructed for all cylinder sets in the topology of $(Q_2,\ell)$, and every open set in this topology can be represented as a union of these cylinder sets, $(Q_2,\ell,\sigma)$ has dense periodic orbits.

    Since $(Q_2,\ell,\sigma)$ is topologically transitive and has dense periodic orbits, and $\sigma: Q_2 \to Q_2$ is a continuous mapping, the Banks-Brooks-Cairns-Davis-Stacey theorem\cite{Banks1992} implies $(Q_2,\ell,\sigma)$ is also sensitive, and therefore chaotic.
\end{proof}
\begin{corollary}
    The $2$-tent map acting on the "standard" unit interval $\mathbf{I}$ is chaotic.
\end{corollary}
\begin{proof}
    Since $(Q_2,\ell,\sigma)$ is transitive and has dense periodic orbits, $([0,1],d^*_2,f_2)$ is also transitive and has dense periodic orbits via conjugacy (Remark \ref{rem:2}), which by the Banks-Brooks-Cairns-Davis-Stacey theorem\cite{Banks1992} indicates that $([0,1],d^*_2,f_2)$ is sensitive and therefore chaotic. And since $\mathfrak{I}_2 = ([0,1],d^*_2)$ is an ultracontinuum of $\mathbf{I} = ([0,1],d)$, $([0,1],d,f_2)$ is also chaotic (Lemma \ref{lem:4}).
\end{proof}
Through further use of conjugacy arguments and the Banks-Brooks-Cairns-Davis-Stacey theorem, one can demonstrate that other classical examples of discrete dynamical systems are mixing and/or chaotic, such as the logistic/Ulam map\cite{ChaosBook}.
\section{Conclusions}
The examples in Sections \ref{sec:mixing} and \ref{chaos} demonstrate that a qualitative framework exists for understanding the dynamics of a continuum that 1) both allows a unique qualitative description for all elements in the continuum and 2) exploits topological conjugacy to prove the existence of mixing and/or chaos everywhere in that continuum. And although these examples demonstrate it for a one-dimensional portion of Euclidean space, one can imagine a similar approach for discrete dynamical systems defined on higher-dimensional continua such as the unit plane or cube that differs little from the schemes described in either section.

The fact that this framework requires the use of sets of infinite symbol sequences that are not within the scope of "standard" symbolic dynamics (see Lemma \ref{lem:1} and its corollary) will hopefully spur interest in dynamicists and combinatorialists on words to study these $\omega$-languages and their use in the field of qualitative dynamics, and to introduce discussions of such systems in introductory texts on the subject.

Similarly, it may be the case that ultracontinua as metric/topological spaces are interesting on their own; although it seems that there is a great deal of leeway in terms of the topological properties of ultracontinua, Lemma \ref{lem:6} demonstrates that there is at least one topological restriction on ultracontinua, and this restriction (their lack of compactness) may make them more pathological than expected. That topological leeway can also easily generate seemingly counter-intuitive properties; for example, the ultracontinua described in Sections \ref{sec:mixing} and \ref{chaos} are zero-dimensional, while the continuum whose topology they contain is one-dimensional (with respect to the notions of Lebesgue covering dimension and both small \& large inductive dimension). 

\begin{acknowledgments}
The author wishes to thank Doğa Yucalan, William Clark, Steven H. Strogatz, and John Guckenheimer for helpful feedback. The author also wishes to acknowledge the American Cancer Society Hope Lodge in New York City, where part of the contents of this paper were conceived/written.
\end{acknowledgments}

\section*{Data Availability Statement}

Data sharing is not applicable to this article as no new data were created or analyzed in this study.

\appendix

\section{Topology \& Topological Properties}

The study of topology begins by endowing some base set $X$ with "structure", characterized by a family of specific subsets of $X$, which (subject to some restrictions) are considered a topology of $X$. Those restrictions are formalized below:
\begin{definition}
Consider a set $X$. A topology $\tau$ of $X$ is a set of subsets of $X$ that satisfies the following three restrictions:
\begin{enumerate}
    \item $X$ and the empty set $\emptyset$ are members of $\tau$.
    \item Any arbitrary union of members of $\tau$ is a member of $\tau$.
    \item The intersection of a finite number of members of $\tau$ is a member of $\tau$.
\end{enumerate}
The $2$-tuple $(X,\tau)$ is collectively considered a topological space. The elements of $\tau$ are called termed the open sets of $(X,\tau)$.
\end{definition}
In this paper, rather than dealing with topological spaces directly, the discussion is centered around the notion of a "metric space" and of a specific topology associated with them:
\begin{definition}
Consider a set $X$. A metric $d$ on $X$ is a mapping that takes any two elements of $X$ as input and outputs a non-negative number representing distance, subject to the following restrictions:
\begin{enumerate}
    \item $d(x,x) = 0$ for any $x \in X$.
    \item $d(x,y) \neq 0$ for any distinct $x, y \in X$.
    \item $d(x,y) = d(y,x)$ for any $x, y \in X$.
    \item $d(x,z) \leq d(x,y) + d(y,z)$ for any $x, y, z \in X$.
\end{enumerate}
The $2$-tuple $(X,d)$ is collectively considered a metric space. 
\newline 

An open ball of $(X,d)$ is a subset of $X$ that can be fully described as "the set of all elements $y \in X$ for which $d(x,y) < \epsilon$ given some $x\in X$ and some $\epsilon > 0$". 
\newline 

The metric topology $\tau_d$ of a metric space $(X,d)$ is the topology generated by the set of all open balls in $(X,d)$ and all possible unions of them.
\end{definition}
Usually (and as is the case in this paper), the topology of a metric space $(X,d)$ is taken for granted to be the metric topology induced by $d$.

There are many properties that one can use to describe topologies and topological spaces. This paper uses two properties whose use is ubiquitous in general topology; compactness and connectedness.
\begin{definition}
Consider a topological space $(X,\tau)$. An open cover $C$ of $X$ is a set of open sets (elements of $\tau$) whose union contains the entirety of $X$. A subcover of $C$ is a subset of $C$ which is still an open cover of $X$.
\newline 

$(X,\tau)$ is compact if and only if every open cover of $(X,\tau)$ with an infinite number of elements (open sets) has a finite subcover.
\end{definition}
There are many notions of compactness that are equivalent on metric spaces (or, to be precise, on topological spaces with a metric topology) which are not relevant to this text; however, the reader is encouraged to read the cited texts\cite{dugundji_1966,steen2013counterexamples,Rudin1976-vd,armstrong2013basic,munkres2000topology} for more information on the role of compactness in topological and metric spaces.

Another key topological property used in this text is the notion of connectedness and disconnectedness:
\begin{definition}
Consider a topological space $(X,\tau)$. $(X,\tau)$ is connected if and only if there do not exist open sets $U, V \in \tau$ such that $U \cup V = X$ and $U \cap V = \emptyset$. If $(X,\tau)$ is not connected, it is said to be disconnected.
\newline 

$(X,\tau)$ is termed totally disconnected if and only if, for any subset $Y \subseteq X$ with more than one element, there exist two open sets $U, V \in \tau$ such that $(U\cap Y)\cup(V\cap Y) = Y$ and $(U\cap Y)\cap(V\cap Y) = \emptyset$.
\end{definition}
Connectedness (and its absence) also plays a key role in the study of general topology; see any of the cited texts\cite{dugundji_1966,steen2013counterexamples,armstrong2013basic,munkres2000topology} for discussions of these properties.

It is also often the case that one wants to develop an understanding of mappings that "warp" topological spaces in a way that "preserves their structure". These mappings are called homeomorphisms:
\begin{definition}
    Consider any two topological spaces $(X_1,d_1)$ and $(X_2,d_2)$. Suppose there existed a bijection, or one-to-one correspondence, $h: X_1 \to X_2$ between the elements of $X_1$ and $X_2$. $h$ is then a homeomorphism if and only if $h$ also defines a bijection between elements of $\tau_1$ and $\tau_2$. 
    \newline
    
    To be more precise, consider the auxiliary mapping $h^\tau$ from $\tau_1$ to the set of all subsets of $X_2$, defined such that $h^\tau(U) = \bigcup_\alpha h(u_\alpha)$ for any $U \in \tau_1$ and every element $u_\alpha \in U$. $h$ is a homeomorphism if and only if $h^\tau$ is a bijection between $\tau_1$ and $\tau_2$.
    \newline
    
    Two topological spaces are considered homeomorphic if and only if there exists a homeomorphism between them. Two metric spaces are considered homeomorphic if and only if there exists a homeomorphism between their underlying sets when equipped with their metric topologies.
\end{definition}
If two topological spaces are homeomorphic, their structures are considered to be "topologically equivalent", in the sense that there exists a one-to-one correspondence between both the underlying set and their respective topologies. Perhaps unsurprisingly, compactness, connectedness, and its absence are all conserved after transforming a topological space with a homeomorphism:
\begin{remark}
\label{rem:7}
    Consider two topological spaces $(X_1,\tau_1)$ and $(X_2,\tau_2)$. If $(X_2,\tau_2)$ is homeomorphic to $(X_1,\tau_1)$ and $(X_1,\tau_1)$ is compact/connected/disconnected/totally disconnected, $(X_2,\tau_2)$ is correspondingly compact/connected/disconnected/totally disconnected.
\end{remark}
Remark \ref{rem:7} has a side-effect which is used centrally in the main discussion of this paper; if two topological spaces have incompatible topological properties, then there cannot exist a homeomorphism between the two topological/metric spaces.
\bibliography{aipsamp}

\end{document}